\documentclass[a4paper, 12pt]{amsart}
\usepackage{amsfonts, amsmath, amsthm}
\usepackage{amssymb}
\usepackage{latexsym}
\usepackage{color}

\theoremstyle{plain}
\newtheorem{thm}{Theorem}[section]

\newtheorem{lem}{Lemma}[section]

\theoremstyle{definition}

\theoremstyle{remark}

\numberwithin{equation}{section}
\newcommand{\R}{\mathbb{R}}
\newcommand{\C}{\mathbb{C}}

\newcommand{\pa}{\partial}
\newcommand{\eps}{\varepsilon}
\newcommand{\jb}[1]{\langle #1 \rangle}

\newcommand{\jbf}[1]{\left\langle #1 \right\rangle}

\newcommand{\op}[1]{\mathcal{#1}}
\newcommand{\cc}[1]{\overline{#1}}

\DeclareMathOperator{\realpart}{\rm Re}
\DeclareMathOperator{\imagpart}{\rm Im}

\newcommand{\dis}{\displaystyle}

\title[Lifespan of solutions to cubic DNLS]
 {
  The lifespan of small solutions to cubic derivative  nonlinear Schr\"odinger 
  equations in one space dimension
 }

\author[Yuji Sagawa]{Yuji Sagawa}
\author[Hideaki Sunagawa]{Hideaki Sunagawa}

\address{Department of Mathematics, Graduate School of Science, 
         Osaka University. %1-1 Machikaneyama-cho, 
         Toyonaka, Osaka 560-0043, Japan.} 
\email{y-sagawa@cr.math.sci.osaka-u.ac.jp}
\email{sunagawa@math.sci.osaka-u.ac.jp}
\date{\today }   



\begin{document}
%%%%%%%%%%%%%%%%%%%%%%%%%%%%%%%%%%%%%%%%%%%%%%%%%%%%%%%%%%%%%%%%%%%%%%%%%%%%%%%

\begin{abstract} 
Consider the initial value problem for 
cubic derivative nonlinear Schr\"odinger equations in one space dimension. 
We provide a detailed lower bound estimate for the lifespan 
of the solution, which can be computed explicitly from the initial data 
and the nonlinear term. 
This is an extension and a refinement of the previous work by one of the 
authors [H.~Sunagawa: Osaka J. Math. {\bf 43} (2006), 771--789] where 
the gauge-invariant nonlinearity was treated. \\
\end{abstract}
\maketitle

%%%%%%%%%%%%%%%%%%%%%%%%%%%%%%%%%%%%%%%%%%%%%%%%%%%%%%%%%%%%%%%%%%%%%%%%%%%%%%%
\section{Introduction and the main result}  \label{sec_intro}

This paper is concerned with the lifespan of solutions to 
cubic derivative nonlinear Schr\"odinger equations in one space dimension 
with small initial data: 
\begin{align}
 \left\{\begin{array}{cl}
  i\pa_t u +\frac{1}{2}\pa_x^2u = N(u,\pa_x u), & t>0,\ x \in \R,\\
  u(0,x)=\eps \varphi(x), &x \in \R,
\end{array}\right.
\label{nls}
\end{align}
where $i=\sqrt{-1}$, $u=u(t,x)$ is a $\C$-valued unknown function, 
$\eps >0$ is a small parameter which is responsible for the size of 
the initial data, and $\varphi$ is  a prescribed $\C$-valued function 
which belongs to $H^3\cap H^{2,1}(\R)$. 
Here and later on as well, $H^s$ denotes the standard $L^2$-based Sobolev 
space of order $s$, and the weighted Sobolev space $H^{s,\sigma}$ is defined 
by $\{\phi \in L^2\,|\, \jb{\, \cdot \,}^{\sigma} \phi \in H^s\}$, 
equipped with the norm 
$\|\phi\|_{H^{s,\sigma}}=\|\jb{\, \cdot \,}^{\sigma} \phi \|_{H^s}$, 
where $\jb{x}=\sqrt{1+x^2}$. 
Throughout this paper, the nonlinear term $N(u,\pa_x u)$ is always assumed to 
be a cubic homogeneous polynomial in  $(u,\cc{u}, \pa_x u, \cc{\pa_x u})$ with 
complex coefficients.  We will often write $u_x$ for $\pa_x u$.

From the perturbative point of view, 
cubic nonlinear Schr\"odinger equations in 
one space dimension are of special interest because the best possible 
decay in $L_x^2$ of general cubic nonlinear terms is $O(t^{-1})$, 
so the cubic nonlinearity 
must be regarded as a long-range perturbation. In general, standard 
perturbative approach is valid only for $t\lesssim \exp(o(\eps^{-2}))$, and 
our problem is to make clear how the nonlinearity affects the behavior of the 
solutions for $t\gtrsim \exp(o(\eps^{-2}))$. 
Let us recall some known results briefly. 
The most well-studied case is the gauge-invariant case, that is the case where 
$N$ satisfies 
\begin{align}
 N(e^{i\theta}z,e^{i\theta}\zeta) 
 =
 e^{i \theta} N(z,\zeta), \qquad (z,\zeta)\in \C\times \C,\quad \theta \in \R.
\label{cond_gi}
\end{align}
There are a lot of works devoted to large-time behavior of the solution to 
\eqref{nls} under \eqref{cond_gi} (see e.g., \cite{Tsu}, 
\cite{KT}, \cite{O}, \cite{HN002}, \cite{HN003}, \cite{HNU}, 
\cite{HNS} and the references cited therein). On the other hand, 
if \eqref{cond_gi} is violated, the situation becomes delicate due to 
the appearance of oscillation structure. It is pointed out in \cite{HN2} 
(see also  \cite{HN1}, \cite{Tone}, \cite{N}) that 
contribution of non-gauge-invariant terms  may be regarded as a short-range 
perturbation if at least one derivative of $u$ is included, 
whereas, as studied in \cite{HN004}, \cite{HN005}, \cite{HN006}, \cite{HN007}, 
\cite{N2}, \cite{HN008} etc., it turns out that contribution of 
non-gauge-invariant cubic terms without derivative is quite difficult to 
handle. In what follows, let us assume that $N$ satisfies 
\begin{align}
 N(e^{i\theta},0) =e^{i \theta} N(1,0), \qquad \theta \in \R,
\label{cond_weak}
\end{align}
to exclude the worst terms $u^3$, $\cc{u}^2 u$ and $\cc{u}^3$ 
(see the appendix for explicit representation of $N$ 
satisfying \eqref{cond_weak}). We also define $\nu:\R\to \C$ by 
\[
 \nu(\xi):=\frac{1}{2\pi i} \oint_{|z|=1} N(z,i\xi z) \frac{dz}{z^2},
 \qquad \xi \in \R.
\]
Roughly speaking, this contour integral extracts the contribution of 
the gauge-invariant part in $N$. Remark that $\nu(\xi)$ coincides with 
$N(1,i\xi)$ in the gauge-invariant case 
(see also \eqref{nu_explicit} below). 
Typical previous results on global existence and large-time asymptotic 
behavior of solutions to \eqref{nls} under \eqref{cond_weak} 
can be summarized in terms of $\nu(\xi)$ as follows 
(see \cite{HN2} and \cite{HNS} for the detail): 
\begin{itemize}
\item[(i)] 
If $\imagpart \nu(\xi) \le 0$ for all $\xi \in \R$, then the solution exists 
globally  in $C([0,\infty);H^3 \cap H^{2,1}(\R))$ for sufficiently small 
$\eps$. Moreover the solution satisfies
\[
 \|u(t)\|_{L^{\infty}} \le \frac{C\eps}{\sqrt{1+t}}, \quad t\ge 0,
\]
where the constant $C$ is independent of $\eps$.

\item[(ii)]
If $\imagpart \nu(\xi)=0$ for all $\xi \in \R$, then the solution has 
a logarithmic oscillating factor in the asymptotic profile, 
i.e., it holds that 
\[
 u(t,x)= \frac{1}{\sqrt{t}} \alpha(x/t) 
 \exp\left( 
  i\frac{x^2}{2t} - i|\alpha(x/t)|^2 \realpart \nu (x/t) \log t
 \right)
 + 
o(t^{-1/2})
\]
as $t \to +\infty$ uniformly in $x \in \R$, 
where  $\alpha(\xi)$ is a suitable $\C$-valued function satisfying 
$|\alpha(\xi)| \lesssim \eps$.
In particular, the solution is asymptotically free if and only if 
$\nu (\xi)$ vanishes identically on $\R$.

\item[(iii)]
If $\sup_{\xi \in \R} \imagpart \nu(\xi)<0$, then the solution gains an 
additional logarithmic time-decay:
\[
 \|u(t)\|_{L^{\infty}} \le \frac{C'\eps}{\sqrt{(1+t)\{ 1+\eps^2 \log(2+t)\}}}, 
 \quad t\ge 0, 
\]
where the constant $C'$ is independent of $\eps$. 
\end{itemize}

Now, let us turn our attentions to the remaining case: 
$\imagpart \nu(\xi_0)>0$ for some $\xi_0 \in \R$. 
To the authors' knowledge, there is no global existence result in that case, 
and many interesting problems are left unsolved especially 
when we focus on the issue of small data blow-up. 
In the previous paper \cite{Su}, 
lower bounds for the lifespan $T_{\eps}$ of the solution to \eqref{nls} 
are considered in detail under the assumption \eqref{cond_gi}. 
It is proved in \cite{Su} that 
\begin{align}
 \liminf_{\eps \to +0} \eps^2 \log T_{\eps} 
\ge 
\frac{1}{\dis{2\sup_{\xi \in \R}(|\hat{\varphi}(\xi)|^2 \imagpart N(1,i\xi))}}
=:\tau_0,
\label{est_Teps_gi}
\end{align}
where $\hat{\varphi}$ denotes the Fourier transform of $\varphi$, i.e.,
\[
 \hat{\varphi}(\xi)
 =
 \frac{1}{\sqrt{2\pi}} \int_{\R} e^{-iy\xi} \varphi(y)\, dy,\quad \xi \in \R,
\]
by constructing an approximate solution $u_a$ which blows up at the time 
$t=\exp(\tau_0/\eps^2)$ and getting an a priori estimate not for the solution 
$u$ itself but for the difference $u-u_a$. 
What is important in \eqref{est_Teps_gi} is that this is quite analogous to the 
famous results due to John \cite{J} and H\"ormander \cite{Hor1} which concern 
quasilinear wave equations in three space dimensions (see also \cite{Hor2} and 
\cite{De} for related results on the Klein-Gordon case). Remember that 
the detailed lifespan estimates obtained in \cite{J} and \cite{Hor1} are 
fairly sharp and have close connection with the so-called {\em null condition} 
introduced by Klainerman \cite{Kla} and Christodoulou \cite{Chr}. 
However, the approach exploited in \cite{Su} has the following two drawbacks: 
\begin{itemize}
\item
it heavily relies on the gauge-invariance \eqref{cond_gi},
\item
it requires higher regularity and faster decay as $|x| \to \infty$ for 
$\varphi$ than those for $u(t,\cdot)$.
\end{itemize}

The purpose of this paper is to improve these two points. 
To state the main result, let us define $\tilde{\tau}_0 \in (0,+\infty]$ 
by
\begin{align}
 \frac{1}{\tilde{\tau}_0}
 =
 2\sup_{\xi \in \R}\Bigl(|\hat{\varphi}(\xi)|^2 \imagpart \nu(\xi) \Bigr),
\label{def_tauzero}
\end{align}
where we associate $1/\tilde{\tau}_0 =0$ with $\tilde{\tau}_0 =+\infty$. 
Remark that the right-hand side of \eqref{def_tauzero} is always non-negative 
because $\imagpart \nu(\xi)=O(|\xi|^3)$ and 
$|\hat{\varphi}(\xi)|^2=O(|\xi|^{-4})$ as $|\xi|\to \infty$. 
In particular, we can easily check that $\tilde{\tau}_0=+\infty$ if 
$\imagpart \nu(\xi)\le 0$ for all  $\xi \in \R$. 
We note also that $\tilde{\tau}_0$ coincides with 
$\tau_0$ if \eqref{cond_gi} is satisfied. 
The main result of this paper is as follows: 

%-------------------------
\begin{thm}\label{thm_main}
Assume that $\varphi \in H^3 \cap H^{2,1}(\R)$. 
Suppose that the nonlinear term $N$ satisfies \eqref{cond_weak}. 
Let $T_{\eps}$ be the supremum of $T>0$ such that \eqref{nls} admits a 
unique solution in $C([0,T);H^3 \cap H^{2,1}(\R))$. 
Then we have 
\[
 \liminf_{\eps \to +0} \eps^2 \log T_{\eps} 
\ge 
\tilde{\tau}_0,
\]
where $\tilde{\tau}_0 \in (0,+\infty]$ is given by \eqref{def_tauzero}. 
\end{thm}
%--------------------------

We close this section with the contents of this paper: 
Section~\ref{sec_ode} is devoted to a lemma on some ordinary 
differential equation. 
In Section~\ref{sec_prelin}, we recall basic properties of the operators 
$\op{J}$ and $\op{Z}$, as well as the smoothing property of the linear 
Schr\"odinger equations. After that, we will get an a priori estimate in 
Section~\ref{sec_apriori}, and the main theorem will be proved 
in Section~\ref{sec_proof}. The proof of technical lemmas will be given in 
the appendix.

%%%%%%%%%%%%%%%%%%%%%%%%%%%%%%%%%%%%%%%%%%%%%%%%%%%%%%%%%%%%%%%%%%%%%%%%%%%%%%%
\section{A lemma on ODE}  \label{sec_ode}

In this section we introduce a lemma on some ordinary 
differential equation, keeping in mind an  application to \eqref{ode_profile} 
below. 

Let $\kappa$, $\theta_0:\R\to \C$ be continuous functions satisfying
\[
 \sup_{\xi\in \R} |\kappa(\xi)| <\infty, \quad 
 \sup_{\xi\in \R} |\theta_0(\xi)|<\infty,\quad 
 \sup_{\xi\in \R}\bigl( |\theta_0 (\xi)|^2 \imagpart\kappa(\xi) \bigr)\ge 0.
\]
We set $C_1=\sup_{\xi\in \R} |\kappa(\xi)|$ and 
define $\tau_1 \in (0,+\infty]$ by
\[
 \tau_1 
 =
\frac{1}{2 \sup_{\xi\in \R}\bigl(|\theta_0 (\xi)|^2 \imagpart\kappa(\xi) \bigr)},
\]
where $1/0$ is understood as $+\infty$. Let $\beta_0(t,\xi)$ be a solution to 
\begin{align}
 \left\{\begin{array}{l}
 \dis{i\pa_t \beta_0(t,\xi)=\frac{\kappa(\xi)}{t} |\beta_0
 (t,\xi)|^2 \beta_0(t,\xi),}\\[3mm]
 \beta_0(1,\xi) =\eps \theta_0(\xi),
 \end{array}\right.
\label{ode_unperturbed}
\end{align}
where $\eps >0$ is a parameter. Then it is easy to see that
\[
 |\beta_0(t,\xi)|^2 
 = 
 \frac{\eps^2 |\theta_0(\xi)|^2}
 {1-2\eps^2 |\theta_0(\xi)|^2 \imagpart \kappa(\xi) \log t}
\]
as long as the denominator is strictly positive. 
In view of this expression, we can see that
\begin{align}
 \sup_{(t,\xi)\in [1,e^{\sigma/\eps^2}]\times \R} |\beta_0(t,\xi)| 
 \le C_2 \eps 
\label{est_unperturbed}
\end{align}
for $\sigma \in (0,\tau_1)$, where
\[
 C_2=\frac{1}{\sqrt{1-\sigma/\tau_1}}\, \sup_{\xi \in \R} |\theta_0(\xi)|
 \ \ (<\infty),
\]
while 
\[
  \sup_{\xi \in \R} |\beta_0(t,\xi)| \to +\infty
  \ \  \mbox{as \ $t\to \exp(\tau_1/\eps^2)-0$}
\]
if $\tau_1<\infty$.

Next we consider a perturbation of \eqref{ode_unperturbed}. 
For this purpose, 
let $T>1$ and let $\theta_1:\R\to \C$, $\rho:[1,T)\times \R\to \C$ be 
continuous functions satisfying
\[
 \sup_{\xi \in \R} |\theta_1(\xi)| \le C_3 \eps^{1+\delta}, 
 \qquad
 \sup_{(t,\xi)\in [1,T)\times \R} t^{1+\mu} |\rho(t,\xi)| 
 \le 
 C_4\eps^{1+\delta}
\]
with some $C_3$, $C_4$, $\delta$, $\mu>0$. 
Let $\beta:[1,T)\times \R \to \C$ be a smooth function satisfying
\[
 \left\{\begin{array}{lc}
 \dis{i\pa_t \beta(t,\xi)
 =\frac{\kappa(\xi)}{t} |\beta(t,\xi)|^2 \beta(t,\xi) 
  +\rho(t,\xi),} 
 & (t,\xi)\in (1,T)\times \R, 
 \\[3mm]
 \beta(1,\xi) =\eps \theta_0(\xi) +\theta_1(\xi). &
 \end{array}\right.
\]
The following lemma asserts that an estimate similar to \eqref{est_unperturbed} remains valid if \eqref{ode_unperturbed} is perturbed by $\rho$ 
and $\theta_1$: 

%-----------------------
\begin{lem} \label{lem_ode}
Let $\sigma \in (0,\tau_1)$. We set 
$T_*=\min\{ T, e^{\sigma/\eps^2}\}$. 
For $\eps \in (0, M^{-1/\delta}]$, we have 
\[
  \sup_{(t,\xi)\in [1,T_*)\times \R} |\beta(t,\xi)| 
 \le 
 C_2\eps +M\eps^{1+\delta},
\]
where
\[
 M= \left(2C_3 + \frac{C_4}{\mu} \right)e^{C_1(1+3C_2+3C_2^{2})\sigma}.
\]

\end{lem}
%-----------------------

\begin{proof}
We put $w(t,\xi)=\beta(t,\xi)-\beta_0(t,\xi)$ and
\[
 T_{**} =\sup \Bigl\{ \tilde{T} \in [1,T_{*})\, \Bigm|\, 
 \sup_{(t,\xi) \in [1,\tilde{T})\times \R} |w(t,\xi)| \le M\eps^{1+\delta} 
 \Bigr\}.
\]
Note that $T_{**} >1$, because of the estimate 
\[
 |w(1,\xi)| 
 =
 |\theta_1(\xi)| 
 \le 
 C_3 \eps^{1+\delta} 
 \le 
 \frac{M}{2}\eps^{1+\delta}
\] 
and the continuity of $w$. Since $w$ satisfies
\[
 i\pa_t w=\frac{\kappa(\xi)}{t}
 \Bigl(
  |w+\beta_0|^2(w+\beta_0) -|\beta_0|^2\beta_0
 \Bigr) 
+\rho,
\]
we see that 
\begin{align*}
 \pa_t|w|^2
 &=
 2\imagpart \Bigl(\cc{w} \cdot i\pa_t w \Bigr)\\
 &\le
 \frac{2}{t}
 C_1 \Bigl( M^2\eps^{2+2\delta} +3 C_2 M\eps^{2+\delta} + 3C_2^2 \eps^2
 \Bigr)  |w|^2
 +|w||\rho|\\
 &\le 
 \frac{2}{t} \tilde{C}\eps^2 |w|^2 +\frac{C_4\eps^{1+\delta}}{t^{1+\mu}}|w|
 \end{align*}
for $t \in [1,T_{**})$, where $\tilde{C}=C_1(1+3C_2+3C_2^2)$. 
By the Gronwall-type argument, we obtain 
\begin{align*}
 |w(t,\xi)| 
 &\le 
 \left(|\theta_1(\xi)| 
 + 
 \int_{1}^{\infty} \frac{C_4\eps^{1+\delta}}{2s^{1+\mu+\tilde{C}\eps^2}} ds
 \right) e^{\tilde{C}\eps^2 \log t}\\
 &\le
 \left(
 C_3 \eps^{1+\delta} + \frac{C_4 \eps^{1+\delta}}{2(\mu+\tilde{C}\eps^2)}
 \right) e^{\tilde{C} \sigma}\\
 &\le
 \frac{M}{2} \eps^{1+\delta}
\end{align*}
for $(t,\xi)\in [1,T_{**})\times\R$. 
This contradicts the definition of $T_{**}$ if $T_{**} <T_{*}$. 
Therefore we conclude $T_{**} =T_*$. In other words, we have 
\[
 \sup_{(t,\xi) \in [1,T_*)\times \R} |w(t,\xi)| \le M\eps^{1+\delta},
\]
whence
\[
 |\beta(t,\xi)|
 \le 
 |\beta_0(t,\xi)|+|w(t,\xi)| 
 \le
 C_2 \eps +M \eps^{1+\delta} 
\]
for $(t,\xi)\in [1,T_*)\times\R$. This completes the proof.
\end{proof}

%%%%%%%%%%%%%%%%%%%%%%%%%%%%%%%%%%%%%%%%%%%%%%%%%%%%%%%%%%%%%%%%%%%%%%%%%%%%%%%
\section{Preliminaries related to the Schr\"odinger operator}  
\label{sec_prelin}

This section is devoted to preliminaries related to the operator 
$\op{L}=i\pa_t +\frac{1}{2}\pa_x^2$. In what follows, 
we denote several positive constants by $C$, which may vary 
from one line to another.

%-------------------------------------------------------------------
\subsection{The operators $\op{J}$ and $\op{Z}$}  \label{sec_vec}

We introduce $\op{J}=x+it\pa_x$ and $\op{Z}=x\pa_x +2t\pa_t$, 
which have good compatibility with $\op{L}$. The following relations 
will be used repeatedly in the subsequent sections: 
\[
 [\pa_x, \op{J}]=1,\quad
 [\op{L}, \op{J}]=0,\quad
 [\pa_x, \op{Z}]=\pa_x, \quad
 [\op{L},\op{Z}]=2\op{L},
\]
where $[\cdot, \cdot]$ stands for the commutator of two linear operators. 
Another important relation is 
\begin{align}
 \op{J}\pa_x =\op{Z} +2it\op{L},
 \label{id_J_and_Z}
\end{align}
which will play the key role in our analysis. 
Next we set 
\[
 (\op{U}(t) \phi)(x)
 =
 e^{i\frac{t}{2}\pa_x^2} \phi(x)
 =
 \frac{e^{-i\frac{\pi}{4}}}{\sqrt{2\pi t}} 
 \int_{\R} e^{i\frac{(x-y)^2}{2t}} \phi(y) dy
\]
for $t> 0$. 
We will occasionally abbreviate $\op{U}(t)$ to $\op{U}$ 
if it causes no confusion. 
Also we introduce 
\[
 (\op{G} \phi)(\xi)
 =
 e^{-i\frac{\pi}{4}} \hat{\phi}(\xi)
 =
 \frac{e^{-i\frac{\pi}{4}}}{\sqrt{2\pi}} 
 \int_{\R} e^{-iy\xi} \phi(y) dy.
\]
The following lemma is well-known 
(see the series of papers by Hayashi and Naumkin 
\cite{HN002}--\cite{HN008} for the proof):

%--------------
\begin{lem} \label{lem_pointwise}
We have
\begin{align*}
\|v\|_{L^{\infty}} 
\le t^{-1/2} \left\| \op{G}\op{U}^{-1} v\right\|_{L^{\infty}} 
+Ct^{-3/4}(\|v\|_{L^2} +\|\op{J} v\|_{L^2})
\end{align*}
for $t>0$. 
\end{lem}
%--------------

%-------------------------------------------------------------------
\subsection{Smoothing property}  \label{sec_smoothing}
In this subsection, we recall smoothing properties of the linear Schr\"odinger 
equations, which will be used effectively in Step 3 of 
\S \ref{sec_L2_est}. 
Among various kinds of smoothing properties, we will follow the approach of 
\cite{HNP}. Let $\op{H}$ be the Hilbert transform, that is, 
\[
\op{H}v(x):=\frac{1}{\pi} \mathrm{p.v.} \int_{\R} \frac{v(y)}{x-y}\,dy.
\]
With a non-negative weight function $\Phi(x)$, let us define the 
operator $\op{S}_{\Phi}$ by 
\[
\op{S}_{\Phi}v(x)
:=
\left\{ \cosh\left(\int_{-\infty}^{x} \Phi (y)\,dy\right)\right\}v(x)
-i\left\{ \sinh\left(\int_{-\infty}^{x} \Phi (y)\,dy\right)\right\}\op{H} v(x).
\]
Note that $\op{S}_{\Phi}$ is $L^2$-automorphism and both 
$\|\op{S}_{\Phi}\|_{L^2\to L^2}$ and $\|\op{S}_{\Phi}^{-1}\|_{L^2\to L^2}$ are 
dominated by $C\exp(\|\Phi\|_{L^1})$. 
The following two lemmas enable us to get rid of the derivative loss 
coming from the nonlinear term:

%----------------------
\begin{lem}\label{lem_smoothing}
Let $v=v(t,x)$ and $\psi=\psi(t,x)$ be $\C$-valued smooth functions. We set 
$\Phi (t,x)=\eta (|\psi|^2+|\psi_x|^2)$ with $\eta \geq 1$. 
Then there exists the constant $C$, which is independent of $\eta$, such that
\begin{align*}
 & \frac{d}{dt}\|\op{S}_{\Phi} v(t)\|_{L^2}^2
  +\left\|\sqrt{\Phi (t)}\op{S}_{\Phi}|\pa_x|^{1/2}v(t)\right\|_{L^2}^2 \\
 &\le 
  Ce^{C\eta \|\psi (t) \|_{H^1}^2} 
 \left(
   \eta \|\psi(t)\|_{W^{2,\infty}}^2 +\eta ^3\|\psi(t)\|_{W^{1,\infty}}^6 
   +\eta \|\psi(t)\|_{H^1} \|\op{L} \psi(t)\|_{H^1}
 \right)
 \|v(t)\|_{L^2}^2\\
 &\quad +
 2\left| 
  \jbf{\op{S}_{\Phi}v(t), \op{S}_{\Phi}\op{L} v(t)}_{L^2}
 \right|,
\end{align*}
where we denote by $W^{k,\infty}$ the $L^{\infty}$-based Sobolev 
space of order $k$.
\end{lem}
%-------------------------

%--------------------------
\begin{lem}\label{lem_aux}
Let $v=v(x)$ and $\psi=\psi(x)$ be $\C$-valued smooth functions. 
Suppose that $q_1$ and $q_2$ are quadratic homogeneous polynomials in 
$(\psi,\cc{\psi},\psi_x,\cc{\psi_x})$. 
We set $\Phi (x)=\eta (|\psi|^2+|\psi_x|^2)$ with $\eta \geq 1$. 
Then there exists the constant $C$, which is independent of $\eta$, such that
\begin{align*}
 & \left| 
     \jbf{\op{S}_{\Phi}v, \op{S}_{\Phi}q_1(\psi,\psi_x) \pa_x v}_{L^2}
   \right|
   +
   \left| 
    \jbf{\op{S}_{\Phi}v, \op{S}_{\Phi}q_2(\psi,\psi_x) \cc{\pa_x v}}_{L^2}
   \right| \\
 &\le 
  C\eta^{-1} e^{C\eta \|\psi\|_{H^1}^2} 
    \left\|\sqrt{\Phi} \op{S}_{\Phi}|\pa_x|^{1/2}v \right\|_{L^2}^2\\
 &\quad +
  C e^{C\eta \|\psi\|_{H^1}^2} 
    \left(
      1+\eta^2 \|\psi\|_{H^1}^4 + \eta^2 \|\psi\|_{W^{1,\infty}}^4  
    \right) \|\psi\|_{W^{2,\infty}}^2 \|v\|_{L^2}^2.
\end{align*}
\end{lem}
%------------------------

For the proof, see Section 2 in \cite{HNP} 
(see also the appendix of \cite{LS}).

%%%%%%%%%%%%%%%%%%%%%%%%%%%%%%%%%%%%%%%%%%%%%%%%%%%%%%%%%%%%%%%%%%%%%%%%%%%%%%%
\section{A priori estimate}  \label{sec_apriori}

Throughout this section, we fix $\sigma \in (0,\tilde{\tau}_0)$ and 
$T \in (0,e^{\sigma/\eps^2}]$, 
where $\tilde{\tau}_0$ is defined by \eqref{def_tauzero}. 
Let $u \in C([0,T);H^3\cap H^{2,1})$ 
be a solution to \eqref{nls} for $t\in [0,T)$, and we set
$\alpha(t,\xi)=\op{G} \bigl[\op{U}(t)^{-1}u(t,\cdot)\bigr](\xi)$, 
where $\op{G}$ and $\op{U}$ are given in Section \ref{sec_prelin}.
We also put
\[
 E(T)=\sup_{t\in [0,T)} \Bigl(
 (1+t)^{-\gamma}(\|u(t)\|_{H^3} +\|\op{J}u(t)\|_{H^2}) +
  \sup_{\xi \in \R} (\jb{\xi}^2|\alpha(t,\xi)|)
 \Bigr)
\]
with $\gamma\in (0,1/12)$. The goal of this section is to prove the following: 

%----------------------
\begin{lem} \label{lem_apriori}
Assume that $N$ satisfies \eqref{cond_weak}. 
Let $\sigma$, $T$ and $\gamma$ be as above. 
Then there exist positive constants $\eps_0$ and $K$, not depending on $T$, 
such that 
\begin{align}
  E(T)\le \eps^{2/3} 
\label{assump}
\end{align}
implies
\[
 E(T)\le K\eps,
\]
provided that $\eps \in (0,\eps_0].$
\end{lem}
%----------------------

We divide the proof of this lemma into two subsections. 
We remark that many parts of the proof below 
are similar to that of Section 3 in 
\cite{HN2}, although we need modifications to fit for our purpose.

%-----------------------------
\subsection{$L^2$-estimates} \label{sec_L2_est}
In this part, we consider the bound for $\|u(t)\|_{H^3} + \|\op{J}u(t)\|_{H^2}$. By virtue of the inequality
\begin{align}
 \|u(t)\|_{H^3} + \|\op{J}u(t)\|_{H^2}
 \le 
 C\bigl(
    \|u(t)\|_{L^2} + \|\pa_x^3 u(t)\|_{L^2} 
  + \|\op{J} u(t)\|_{L^2}  + \|\pa_x \op{J} \pa_x u(t)\|_{L^2}
 \bigr),
\label{est_L_two_0}
\end{align}
it suffices to show that each term in the right-hand side can be dominated by 
$C\eps (1+t)^{\gamma}$. 
We are going to estimate these four terms by separate ways. 
\\

%--------------------------------------------------
\noindent \underline{\textbf{Step 1}:\  Estimate for $\|u(t)\|_{L^2}$.} 
First we remark  that  \eqref{assump} yields
\[
 \|u(t)\|_{W^{2,\infty}} \le \frac{C\eps^{2/3}}{(1+t)^{1/2}}
\]
for $t \in [0,T)$. Indeed, the Sobolev embedding 
$H^{1}(\R)\hookrightarrow L^{\infty}(\R)$ yields 
\[
\|u(t)\|_{W^{2,\infty}} \le C\|u(t)\|_{H^3} \le C\eps^{2/3}
\] 
for $t\le 1$, while it follows from Lemma \ref{lem_pointwise} that 
\[
 \|u(t)\|_{W^{2,\infty}} 
 \le 
 \frac{C}{t^{1/2}}  \sup_{\xi \in \R}\bigl( \jb{\xi}^2 |\alpha(t,\xi)|\bigr)
 +
 \frac{C}{t^{3/4}} \bigl( \|u(t)\|_{H^2} + \|\op{J} u(t)\|_{H^2} \bigr)
 \le
 \frac{C\eps^{2/3}}{t^{1/2}}
\]
for $t\in [1,T)$. 
Now, by the standard energy method, we have 
\[
 \frac{d}{dt}\|u(t)\|_{L^2} \le 
 \|N(u,u_x)\|_{L^2}
\le
 C\|u(t)\|_{W^{1,\infty}}^2 \|u(t)\|_{H^1}
 \le
 \frac{C\eps^2}{(1+t)^{1-\gamma}},
\]
whence
\begin{align}
\label{est_L_two_1}
 \|u(t)\|_{L^2} 
 \le& 
 \|u(0)\|_{L^2} +\int_{0}^{t} \frac{C\eps^2}{(1+\tau)^{1-\gamma}}d\tau
 \\
 \le&
 C\eps + C\eps^2 (1+t)^{\gamma}
 \nonumber\\
 \le &
 C\eps (1+t)^{\gamma}
  \nonumber
\end{align}
for $t\in [0,T)$. 
\\

%--------------------------------------------------
\noindent \underline{\textbf{Step 2}:\   Estimate for $\|\op{J}u(t)\|_{L^2}$.} 
If $t\le 1$, there is no difficulty because we do not have to pay attentions 
to possible growth in $t$. Indeed, since  
\[
 \|\op{J}N(u(t),u_x(t))\|_{L^2}
 \le
 C(1+t) \|u\|_{W^{1,\infty}}^2 (\|\op{J}u\|_{H^1} +\|u\|_{H^1})
 \le
 C\eps^2,
\]
we have
\begin{align*}
 \|\op{J}u(t)\| 
 \le 
 C\|xu(0)\|_{L^2} + \int_0^{1} \|\op{J}N(u(\tau),u_x(\tau))\|_{L^2}d\tau
 \le
 C\eps +C\eps^2
\end{align*}
for $t\le 1$. 
To consider the case of $t\ge 1$, let us first recall a 
remarkable lemma due to Hayashi--Naumkin~\cite{HN2}: 
%--------------------------
\begin{lem} \label{lem_decomp}
Assume that $N$ satisfies \eqref{cond_weak}. 
Then the following decomposition holds: 
\[
 \op{J}N(u,u_x)=\op{L}(tP) +Q,
\]
where $P$ is a cubic homogeneous polynomial in $(u,\cc{u},u_x, \cc{u_x})$, 
and $Q$ satisfies
\[
 \|Q\|_{L^2} 
 \le
 C\|u\|_{W^{2,\infty}}^2(\|u\|_{H^1}+\|\op{J}u\|_{H^1} +\|\op{Z} u\|_{H^1})
 +
 \frac{C}{t}\|u\|_{W^{2,\infty}} (\|\op{J}u\|_{H^2}+\|u\|_{H^1})^2
\]
for $t\ge 1$. 
\end{lem}
%------------------------------
For the convenience of the readers, we shall give a sketch of the proof 
in the appendix.  
Now we are going to apply this lemma. Let $t\in [1,T)$. 
Since the above decomposition allows us to rewrite the original equation 
as 
\[
 \op{L}(\op{J}u-tP)=Q,
\]
the standard energy method gives us 
\[
 \|\op{J}u(t)-tP\|_{L^2} 
 \le 
 C(\eps+\eps^2) + \int_{1}^{t} \|Q(\tau)\|_{L^2}d\tau.
\]
By the relation \eqref{id_J_and_Z}, we have 
\[
 \|\op{Z} u\|_{H^1}
 \le 
 \|\op{J}\pa_x u\|_{H^1} + 2t\|N(u, u_x)\|_{H^1}
 \le 
 C\eps (1+t)^{\gamma} +C\eps^2 t (1+t)^{-1+\gamma}
 \le 
 C\eps (1+t)^{\gamma},
 \]
which leads to 
\[
 \|Q(t)\|_{L^2} \le  \frac{C\eps^2 }{(1+t)^{1-\gamma}}.
\]
Also we have
\[
 \|P\|_{L^2} 
\le 
C\|u\|_{W^{1,\infty}}^2 \|u\|_{H^1} \le \frac{C\eps^2}{(1+t)^{1-\gamma}}.
\]
Summing up, we have
\begin{align}
 \|\op{J}u(t)\|_{L^2} 
 \le 
 t\|P\|_{L^2} +\|\op{J}u(t)-tP\|_{L^2}
 \le 
  C\eps (1+t)^{\gamma}
\label{est_L_two_2}
\end{align}
for $t \in [1,T)$. 
\\

%--------------------------------------------------
\noindent \underline{\textbf{Step 3}:\ \  
Estimate for $\|\pa_x^3u(t)\|_{L^2}$}. 
We apply Lemma \ref{lem_smoothing}
with $v=\pa_x^3 u$, $\psi=u$ and $\eta=\eps^{-2/3}$. Then we obtain
\begin{align*}
 \frac{d}{dt} \|\op{S}_{\Phi} \pa_x^3 u(t)\|_{L^2}^2
 +
 \Bigl\|\sqrt{\Phi} \op{S}_{\Phi} |\pa_x|^{1/2} \pa_x^3 u \Bigr\|_{L^2}^2
 &\le
 C B(t) \|\pa_x^3 u \|_{L^2}^2
 +
 2\left| 
 \jbf{\op{S}_{\Phi}\pa_x^3 u, \op{S}_{\Phi}\pa_x^3 N(u,u_x)}_{L^2}
 \right|,
\end{align*}
where
\[
 B(t)=e^{C\eps^{-2/3} \|u\|_{H^1}^2} 
 \left(
   \eps^{-2/3} \|u(t)\|_{W^{2,\infty}}^2 +\eps^{-2} \|u(t)\|_{W^{1,\infty}}^6 
   +\eps^{-2/3} \|u(t)\|_{H^1} \|N(u,u_x)\|_{H^1}\right).
\]
Since \eqref{assump} yields
\[
 \|u(t)\|_{H^1} 
 \le 
 C\|\alpha(t)\|_{H^{0,1}}
 \le 
 C\left(\int_{\R}\frac{d\xi}{\jb{\xi}^2}\right)^{1/2} 
 \sup_{\xi \in \R}\bigl( \jb{\xi}^2 |\alpha(t,\xi)|\bigr)
\le C\eps^{2/3},
\]
we see that $B(t)$ can be dominated by $C\eps^{2/3} (1+t)^{-1}$. 
Also we observe that the usual Leibniz rule leads to
\[
 \pa_x^3 N(u,u_x)
 =
q_1(u,u_x) \pa_x(\pa_x^3 u) + q_2(u,u_x) \pa_x(\cc{\pa_x ^3 u}) + \rho_1, 
\]
where $q_1$, $q_2$ are defined by 
\begin{align}
 q_1(z,\zeta)=\frac{\pa N}{\pa \zeta}(z,\zeta), \qquad 
 q_2(z,\zeta)=\frac{\pa N}{\pa \cc{\zeta}}(z,\zeta),
\label{def_q}
\end{align}
and $\rho_1$ satisfies 
\[
 \|\rho_1\|_{L^2}\le C\|u\|_{W^{2,\infty}}^2 \|u\|_{H^3}. 
\]
By Lemma \ref{lem_aux}, we have
\begin{align*}
 &\left| 
 \jbf{\op{S}_{\Phi}\pa_x^3 u, \op{S}_{\Phi}\pa_x^3 N(u,u_x)}_{L^2}
 \right|\\
 &\le 
 \bigl| \langle
   \op{S}_{\Phi}\pa_x^3 u, \op{S}_{\Phi}q_1(u,u_x) \pa_x(\pa_x^3 u)
 \rangle_{L^2} \bigr|
 + 
 \bigl| \langle
    \op{S}_{\Phi}\pa_x^3 u, \op{S}_{\Phi}q_2(u,u_x) \pa_x(\cc{\pa_x ^3 u})
 \rangle_{L^2} \bigr|
 + 
 C\|\op{S}_{\Phi}\pa_x^3 u\|_{L^2}\|\op{S}_{\Phi}\rho_1\|_{L^2}\\
 &\le
 C\eps^{2/3} e^{C\eps^{-2/3} \|u\|_{H^1}^2} 
 \Bigl\|\sqrt{\Phi} \op{S}_{\Phi}|\pa_x|^{1/2}\pa_x^3 u \Bigr\|_{L^2}^2
 \\
 &\quad 
 +
 C e^{C\eps^{-2/3} \|u\|_{H^1}^2} 
 \left(
  1+ \eps^{-4/3} \|u\|_{H^1}^4 + \eps^{-4/3}\|u\|_{W^{1,\infty}}^4
 \right)
 \|u\|_{W^{2,\infty}}^2
 \|\pa_x^3u\|_{L^2}^2 \\
 &\quad 
 +
  Ce^{C\eps^{-2/3} \|u\|_{H^1}^2}\|u\|_{W^{2,\infty}}^2 \|u\|_{H^3}^2\\
 &\le
 C_0 \eps^{2/3}
 \Bigl\|\sqrt{\Phi} \op{S}_{\Phi} |\pa_x|^{1/2} \pa_x^3 u \Bigr\|_{L^2}^2
 +
 \frac{C\eps^{8/3}}{(1+t)^{1-2\gamma}}
\end{align*}
with some positive constant $C_0$ not depending on $\eps$. 
Piecing the above estimates all together, we obtain
\begin{align}
 \frac{d}{dt} \|\op{S}_{\Phi} \pa_x^3 u(t)\|_{L^2}^2
 \le
 \bigl( 2C_0 \eps^{2/3}-1 \bigr)
 \Bigl\|\sqrt{\Phi} \op{S}_{\Phi} |\pa_x|^{1/2} \pa_x^3 u \Bigr\|_{L^2}^2
 +
 \frac{C(\eps^{2}+\eps^{8/3})}{(1+t)^{1-2\gamma}}
\le
 \frac{C\eps^2}{(1+t)^{1-2\gamma}},
\label{est_pa3}
\end{align}
provided that $\eps \le (2C_0)^{-3/2}$. 
Integrating with respect to $t$, we have 
\[
 \|\op{S}_{\Phi} \pa_x^3 u(t)\|_{L^2}^2 
 \le 
 Ce^{C\eps^{2/3}} \eps^2 +  C \eps^{2}(1+t)^{2\gamma}
 \le 
 C \eps^{2}(1+t)^{2\gamma}, 
\]
whence
\begin{align}
 \|\pa_x^3 u(t)\|_{L^2} 
 \le 
 e^{C\eps^{-2/3} \|u(t)\|_{H^1}^{2}}
 \|\op{S}_{\Phi} \pa_x^3 u(t)\|_{L^2} 
 \le 
 C \eps (1+t)^{\gamma}
\label{est_L_two_3}
\end{align}
for $t \in [0,T)$. 
\\

%--------------------------------------------------
\noindent 
\underline{\textbf{Step 4}:\ Estimate for $\|\pa_x \op{J} \pa_x u(t)\|_{L^2}$}. By using the commutation relation $[\op{L},\pa_x\op{Z}]=2\pa_x \op{L}$ 
and the Leibniz rule for $\op{Z}$, we have
\[
 \op{L}\pa_x \op{Z} u
 =
q_1(u,u_x) \pa_x(\pa_x \op{Z} u) + q_2(u,u_x) \pa_x(\cc{\pa_x \op{Z} u}) 
+ \rho_2, 
\]
where $q_1$, $q_2$ are given by \eqref{def_q}, and $\rho_2$ satisfies
\[
 \|\rho_2\|_{L^2}
 \le 
 C\|u\|_{W^{2,\infty}}^2 (\|u\|_{H^2} + \|\op{Z} u\|_{H^2}).
\]
Since the relation \eqref{id_J_and_Z} leads to 
\[
 \|\op{Z} u\|_{H^2}
 \le 
 \|\op{J}\pa_x u\|_{H^2} + 2t\|N(u,u_x)\|_{H^2}
 \le 
 C\eps (1+t)^{\gamma},
 \]
we see that 
\[
 \|\rho_2\|_{L^2}
 \le 
 \frac{C\eps^2}{(1+t)^{1-\gamma}}. 
\]
Thus, as in the derivation of \eqref{est_pa3}, we have
\[
 \frac{d}{dt}\|\op{S}_{\Phi} \pa_x \op{Z} u(t)\|_{L^2}^2 
 \le 
 \frac{C\eps^2}{(1+t)^{1-2\gamma}}, 
\]
which yields
\begin{align*}
 \|\pa_x \op{Z} u(t)\|_{L^2} 
 \le 
 C \eps (1+t)^{\gamma}. 
\end{align*}
Finally, by using the relation \eqref{id_J_and_Z} again, we obtain
\begin{align}
\label{est_L_two_4}
 \|\pa_x \op{J}\pa_x u(t)\|_{L^2} 
 \le & 
 \| \pa_x\op{Z} u\|_{L^2}+ 2t\|\pa_x N(u,u_x) \|_{L^2} 
 \\
 \le & 
 C \eps (1+t)^{\gamma} 
 + 2t \frac{C\eps^2}{(1+t)^{1-\gamma}}
 \nonumber\\
\le  &
  C \eps (1+t)^{\gamma}
 \nonumber
\end{align}
for $t \in[0,T)$. 
\\

%--------------------------------------------------
\noindent \underline{\textbf{Final step}}. 
Substituting \eqref{est_L_two_1}, \eqref{est_L_two_2}, \eqref{est_L_two_3}, 
\eqref{est_L_two_4} into \eqref{est_L_two_0}, we arrive at the desired estimate
\[
 \|u(t)\|_{H^3} + \|\op{J}u(t)\|_{H^2} 
 \le C\eps (1+t)^{\gamma} 
\] 
for $t \in [0,T)$.

%---------------------------------------------------------------------------
\subsection{Estimates for $\alpha$}\label{sec_Linfty_est} 

In this part, we will show 
$\jb{\xi}^2 |\alpha(t,\xi)| \le C\eps$ 
for $(t,\xi) \in [0,T) \times \R$ under the assumption \eqref{assump}. 
If $t \le 1$, the Sobolev embedding yields this estimate immediately. 
Hence we may assume $T>1$ and $t \in [1,T)$ in what follows. 

Now let us introduce a useful lemma, which is due to Hayashi--Naumkin 
\cite{HN2} though the expression is slightly different. 
We write $\alpha_{\omega}(t,\xi)=\alpha(t,\xi/\omega)$ 
for $\omega \in \R \backslash\{0\}$.

%--------------------------
\begin{lem} \label{lem_decomp2}
Assume that $N$ satisfies \eqref{cond_weak}. Then, for $l\in \{0,1,2\}$, 
the following decomposition holds: 
\begin{align}
\label{id_key}
\op{G}  \op{U}^{-1} \pa_x^l N(u,u_x) 
=&
\frac{(i\xi)^l \nu(\xi)}{t} |\alpha|^2 \alpha 
+
\frac{\xi e^{i\frac{t}{3}\xi^2}}{t} \mu_{1,l}(\xi) \alpha_3^3 
\\
&+
\frac{\xi  e^{i\frac{2t}{3}\xi^2} }{t} \mu_{2,l}(\xi) 
 \bigl(\cc{\alpha_{-3}}\bigr)^3
+
\frac{\xi e^{it\xi^2}}{t} \mu_{3,l}(\xi) |\alpha_{-1}|^2 \cc{\alpha_{-1}}  
+
R_l, 
\nonumber
\end{align}
where $\mu_{1,l}(\xi)$, $\mu_{2,l}(\xi)$, $\mu_{3,l}(\xi)$ are polynomials in $\xi$ of order 
at most $2+l$, and $R_l(t,\xi)$ satisfies
\[
 \sum_{l=0}^{2}\|R_l(t)\|_{L^{\infty}} 
 \le
 \frac{ C}{t^{5/4}} \bigl(\|u(t)\|_{H^3}+\|\op{J}u(t)\|_{H^2} \bigr)^3
\]
for $t\ge 1$.
\end{lem}
%------------------------------
The proof of this lemma will be given in the appendix. 
It follows from this lemma that
\begin{align}
 \label{profile_eq}
 \jb{\xi}^2 i\pa_t \alpha
 &=
 \op{G} \op{U}^{-1}(1-\pa_{x}^2) \op{L}u
 \\
 &=\op{G} \op{U}^{-1} N(u,u_x)-\op{G} \op{U}^{-1}\pa_x^2 N(u,u_x)
 \nonumber\\
 &=
 \frac{\jb{\xi}^2 \nu(\xi)}{t} |\alpha|^2\alpha
 + V+ R_0-R_2,
 \nonumber
\end{align}
where
\[
 V(t,\xi)
 =
 \frac{\xi e^{it\xi^2/3}}{t} p_1(\xi) \alpha_3^3
 + \frac{\xi e^{i2t\xi^2/3}}{t} p_2(\xi) \cc{\alpha_{-3}}^3
 + \frac{\xi e^{it\xi^2}}{t} p_3(\xi) |\alpha_{-1}|^2 \cc{\alpha_{-1}}
\]
with
$p_k(\xi)=\mu_{k,0}(\xi)-\mu_{k,2}(\xi)$ ($k=1,2,3$). 
We deduce from \eqref{assump} and \eqref{profile_eq} that
\[
 |\pa_t \alpha (t,\xi)| \le \frac{C\eps^2}{\jb{\xi}^2 t}, 
 \qquad (t,\xi) \in [1,T)\times \R.
\]
Also, by using the identity
\[
 \frac{\xi  e^{i\omega t\xi^2}}{t} A(t,\xi)
=
 i \pa_t 
 \left(\frac{ -i\xi e^{i\omega t\xi^2}}{1+i \omega t\xi^2} A(t,\xi)\right) 
 -
 t e^{i \omega t\xi^2}
 \pa_t\left( \frac{ \xi A(t,\xi)}{t(1+i\omega t\xi^2)}  \right),
\]
we see that $V$ can be splitted into the following form:
\[
 V =i\pa_t V_1 +V_2, \quad 
 |V_1(t,\xi)| \le \frac{C\eps^2}{t^{1/2}}, \quad
 |V_2(t,\xi)| \le \frac{C\eps^2}{t^{3/2}}.
\]
Note that 
\[
 \sup_{\xi \in \R} 
 \left|\frac{\xi }{1+i \omega t\xi^2} \right| 
 \le 
 \frac{C}{t^{1/2}}
\]
if $\omega \in \R\backslash \{0\}$. 
Now we set 
$\beta(t,\xi)=\jb{\xi}^2\alpha(t,\xi) -V_1(t,\xi)$ 
and $\kappa(\xi)= \jb{\xi}^{-4}\nu(\xi)$. Then we have
\begin{align}\label{ode_profile}
 i\pa_t \beta(t,\xi) 
 =
 \frac{\kappa(\xi)}{t} |\beta(t,\xi)|^2 \beta(t,\xi) +\rho(t,\xi),
\end{align}
where
\[
 \rho(t,\xi)
 =
 \frac{\kappa(\xi)}{t}
 \biggl( \bigl|\jb{\xi}^2 \alpha \bigr|^{2}\jb{\xi}^2 \alpha  
  -|\beta|^2\beta \biggr)
 +V_2(t,\xi)+R_0(t,\xi)-R_2(t,\xi).
\]
Remark that $\rho$ can be regarded as a remainder because we have
\[
 |\rho(t,\xi)|
 \le
 \frac{C}{t} \cdot \bigl(C\eps^{2/3} \bigr)^2 \cdot \frac{C\eps^2}{t^{1/2}}
 +
 \frac{C\eps^2}{t^{3/2}}
 +
 \frac{C}{t^{5/4}}\cdot \bigl(C\eps^{2/3} t^{\gamma}\bigr)^3
\le
 \frac{C \eps^2}{t^{1+\mu}}
\]
with $\mu=1/4-3\gamma>0$. 
Moreover we have
\begin{align*}
\sup_{\xi \in \R} \Bigl| \beta(1,\xi)
 -
 \eps \jb{\xi}^2 \hat{\varphi}(\xi) \Bigr|
&\le  
C \bigl\| 
  (1-\pa_x^2) \bigl(\op{U}(1)^{-1} u(1,\cdot)-\eps \varphi \bigr)
 \bigr\|_{H^{0,1}} 
 + \sup_{\xi \in \R} |V_1(1,\xi)|\\
&\le   
 C \int_0^1 \bigl\|N(u(t),u_x(t)) \bigr\|_{H^{2,1}}dt 
 + 
 C\eps^2\\
&\le
 C \left( \sup_{t\in [0,1]} \|u(t)\|_{H^3\cap H^{2,1}}\right)^3
 + C \eps^2\\
&\le
 C\eps^2.
\end{align*}
Therefore we can apply Lemma \ref{lem_ode} 
with 
$\theta_0(\xi)=\jb{\xi}^2 \hat{\varphi}(\xi)$ 
and $\tau_1=\tilde{\tau}_0$ 
to obtain $|\beta(t,\xi)| \le C\eps$, 
whence
\[
 \jb{\xi}^2|\alpha(t,\xi)|
 \le 
 |\beta(t,\xi)| + |V_1(t,\xi)|
 \le C\eps
\]
for $(t,\xi)\in [1,T)\times \R$, as desired.
\qed

%%%%%%%%%%%%%%%%%%%%%%%%%%%%%%%%%%%%%%%%%%%%%%%%%%%%%%%%%%%%%%%%%%%%%%%%%%%%%%%
\section{Proof of the main theorem}  \label{sec_proof}

Now we are in a position to prove Theorem \ref{thm_main}. 
First we state a standard local existence result without proof. 
Let $t_0 \ge 0$ be fixed, and consider the initial value problem 
\begin{align}
 \left\{\begin{array}{cl}
  \op{L} u = N(u,u_x), & t>t_0,\ x \in \R,\\
  u(t_0,x)=\psi(x), &x \in \R.
\end{array}\right.
\label{nls_loc}
\end{align}

%----------------------
\begin{lem} \label{lem_local}
Let $N$ be a cubic homogeneous polynomial in 
$(u,\cc{u}, u_x, \cc{u_x})$. 
Let $\psi\in H^3\cap H^{2,1}(\R)$. 
Then there exists $T_0=T_0(\|\psi\|_{H^3})>0$, independent of $t_0$, 
such that $\eqref{nls_loc}$ has a unique solution 
$u\in C\left([t_0,t_0+T_0);H^3\cap H^{2,1}(\R)\right)$.
\end{lem}
%----------------------
See \cite{KPV}, \cite{HO}, \cite{Chi}, \cite{KT}, \cite{HNP}, 
etc., for more details on local existence theorems.\\

\noindent
{\em Proof of Theorem \ref{thm_main}}.\ 
Let $T_{\eps}$ be the lifespan 
defined in the statement of Theorem \ref{thm_main}. We remark that  
Lemma \ref{lem_local} with $t_0=0$ and $\psi=\eps \varphi$ implies 
$T_{\eps}>0$. Next we set 
\[
 T^*=\sup\{ T\in [0,T_\eps)\ |\, E(T) \le \eps^{2/3}\}.
\]
Note that $T^*>0$ if $\eps$ is suitably small, because of the estimate 
$E(0)\le C\eps \le (1/2) \eps^{2/3}$ and the continuity of 
$[0,T_{\eps}) \ni T \mapsto E(T)$. 
Now, we take $\sigma \in (0,\tilde{\tau}_0)$ and assume 
$T^*\le e^{\sigma/\eps^2}$. 
Then Lemma \ref{lem_apriori} with $T=T^*$ yields
\[
 E(T^*) \le K\eps \le \frac{1}{2} \eps^{2/3}
\]
if $\eps\le \min\{\eps_0, (2K)^{-3}\}$. By the continuity of 
$[0,T_{\eps}) \ni T \mapsto E(T)$, we can choose $\Delta>0$ such that 
\[
 E(T^*+\Delta) \le \eps^{2/3}.
\]
This contradicts the definition of $T^*$. Therefore we must have
$T^* \ge e^{\sigma/\eps^2}$ if $\eps$ is suitably small. 
Consequently, we have
\[
 \liminf_{\eps \to +0} \eps^2 \log T_{\eps} \ge \sigma.
\]
Since $\sigma \in (0,\tilde{\tau}_0)$ is arbitrary, we arrive at the desired 
conclusion.\qed

%%%%%%%%%%%%%%%%%%%%%%%%%%%%%%%%%%%%%%%%%%%%%%%%%%%%%%%%%%%%%%%%%%%%%%%%%%%%%%%
\appendix\section
{Proof of Lemmas \ref{lem_decomp} and \ref{lem_decomp2}}  

In this appendix, we will prove 
Lemmas \ref{lem_decomp} and \ref{lem_decomp2} along the idea of \cite{HN2}.

%-------------------------------------------
\subsection{Proof of Lemma \ref{lem_decomp}}

First we observe that the nonlinear term $N$ satisfying \eqref{cond_weak} 
can be written as $N=F+G$, where 
\begin{align}
\label{def_F}
 F=
 &a_1 u^2 u_x +a_2 u u_x^2 +a_3 u_x^3
 + b_1 \cc{u^2 u_x} + b_2 \cc{u u_x^2} 
 + b_3 \cc{u_x^3}
 \\
 &+
 c_1 \cc{u^2} u_x +  c_2  |u|^2 \cc{u_x} 
 +
 c_3 u \cc{u_x^2} + c_4 |u_x|^2 \cc{u} + c_5 |u_x|^2 \cc{u_x}
 \nonumber
\end{align}
and
\begin{align}
 G=\lambda_1 |u|^2 u + \lambda_2 |u|^2 u_x +\lambda_3 u^2\cc{u_x}
+
\lambda_4 |u_x|^2 u + \lambda_5 \cc{u} u_x^2 + \lambda_6 |u_x|^2 u_x 
\label{def_G}
\end{align}
with $a_j$, $b_j$, $c_j$, $\lambda_j \in \C$. Note that $G$ is 
gauge-invariant, while $F$ is not. 
By using the identities 
\begin{align}
 \phi  \pa_x \psi
=
 (\pa_x \phi) \psi
 +\frac{1}{it} \bigl( \phi \op{J}\psi-(\op{J}\phi) \psi \bigr)
\label{id_key1}
\end{align}
and 
\begin{align}
\phi \pa_x \cc{\psi}
=
 -(\pa_x \phi) \cc{\psi} + 
 \frac{1}{it} \bigl((\op{J}\phi)\cc{\psi} -\phi\cc{\op{J}\psi}\bigr),
\label{id_key2}
\end{align}
we see that $F$ can be splitted into 
$F=\pa_x F_1 +\frac{1}{it}F_2$, where  
\begin{align*}
 F_1=
&\frac{a_1}{3} u^3 + \frac{a_2}{3} u^2 u_x + \frac{a_3}{3} u_x^2 u 
+ \frac{b_1}{3} \cc{u^3} + \frac{b_2}{3} \cc{u^2 u_x} 
+ \frac{b_3}{3} \cc{u u_x^2}\\
&+
 (c_2-c_1) |u|^2 \cc{u} + c_3 |u|^2 \cc{u_x} 
 + c_4 \cc{u}^2 u_x + c_5 |u_x|^2\cc{u}
\end{align*}
and
\begin{align*}
F_2= 
&\frac{a_2}{3} u \bigl( u\op{J} u_x -u_x \op{J}u \bigr)
 + \frac{2a_3}{3} u_x \bigl( u \op{J}u_x -u_x \op{J}u \bigr)\\
&-\frac{b_2}{3}\cc{u \bigl(  u\op{J} u_x -u_x \op{J}u \bigr)} 
 - \frac{2b_3}{3} \cc{u_x \bigl( u \op{J}u_x -u_x \op{J}u \bigr)}\\
&+(c_2-2c_1)\cc{u} \bigl(u \cc{\op{J}u} -\cc{u} \op{J}u \bigr)
-c_3 \cc{u} \bigl( \cc{u}_x \op{J}u -u_x \cc{\op{J}u_x} \bigr)\\
&-c_4 \cc{u} \bigl( \cc{u}\op{J}u_x - u \cc{\op{J} u} \bigr)
-c_5 \cc{u} \bigl( \cc{u}_x \op{J}u_x - u_x \cc{\op{J}u_x} \bigr).
\end{align*}
We deduce from the relation \eqref{id_J_and_Z} that  
\begin{align*}
 \op{J}N(u,u_x)
 =
 (\op{Z}+2it\op{L})F_1 +\frac{1}{it} \op{J}F_2+ \op{J}G
 =
 \op{L} (tP) + Q,
\end{align*}
where
$P=2iF_1$ and $Q=(\op{Z}+2)F_1 + \frac{1}{it}\op{J}F_2 + \op{J}G$. 
By the Leibniz rule for $\op{Z}$, we have 
\[
 \|(\op{Z} +2)F_1\|_{L^2} 
 \le 
 C\|u\|_{W^{1,\infty}}^2\bigl( \|\op{Z} u\|_{H^1} + \|u\|_{H^1} \bigr).
\]
On the other hand, since $G$ is gauge-invariant, 
we can use the identity
\[
 \op{J}(f_1 f_2\cc{f_3}) =(\op{J}f_1)f_2\cc{f_3} +f_1(\op{J}f_2)\cc{f_3}
 -f_1f_2\cc{\op{J}f_3}
\]
to obtain
\[
 \|\op{J}G\|_{L^2} 
 \le 
 C\|u\|_{W^{1,\infty}}^2\bigl( \|\op{J}u\|_{H^1} + \|u\|_{L^2} \bigr).
\]
As for $F_2$, we see that 
\[
 \|\op{J}F_2\|_{L^2} 
 \le 
 C t \|u\|_{W^{2,\infty}}^2\bigl( \|\op{J} u\|_{H^1} + \|u\|_{L^2} \bigr)
 + C  \|u\|_{W^{2,\infty}}\bigl( \|\op{J}u\|_{H^2} + \|u\|_{H^1} \bigr)^2.
\]
Piecing them together, we arrive at the desired decomposition.\qed

%-------------------------------------------------------------------------
\subsection{Proof of Lemma \ref{lem_decomp2}}

Before we proceed to the proof of Lemma \ref{lem_decomp2}, we introduce 
some notations. We put
\[
 \bigl( \op{M}(t) \phi \bigr)(x)
 =
 e^{i\frac{x^2}{2t}}\phi(x),
 \quad
 \bigl( \op{D}(t) \phi \bigr)(x) 
 = 
 t^{-1/2} \phi \left(\frac{x}{t} \right),
 \quad 
  \op{V}(t)  \phi 
 =
\op{G} \op{M}(t) \op{G}^{-1} \phi,
\]
so that $\op{U}(t)$ is decomposed into 
$\op{U}(t)=\op{M}(t) \op{D}(t) \op{G} \op{M}(t)
=\op{M}(t) \op{D}(t) \op{V}(t) \op{G}$. 
Note that 
\begin{align}
 \|(\op{V}(t)  -1)\phi\|_{L^{\infty}} 
+
 \|(\op{V}(t)^{-1} -1)\phi\|_{L^{\infty}} 
 \le 
 Ct^{-1/4}\|\phi\|_{H^1},
\label{est_W}
\end{align}
which comes from the inequalities 
$|e^{i\theta} -1|\le C|\theta|^{1/2}$
and 
$\|\phi\|_{L^{\infty}}
 \le 
Ct^{-1/2} \|\phi\|_{L^2}^{1/2}\|\op{J}\phi\|_{L^2}^{1/2}$. 
In what follows, we will occasionally omit ``$(t)$" 
from $\op{M}(t)$, $\op{D}(t)$, $\op{V}(t)$ 
if it causes no confusion, and we will write 
$\op{D}_{\omega}=\op{D}(\omega)$ for $\omega\in \R \backslash\{0\}$. 

%----------------
\begin{lem}  \label{lem_tech2}
We have 
\[
 \|\op{G} \op{U}^{-1} [ f_1f_2f_3 ] \|_{L^{\infty}}
 +
 \|\op{G} \op{U}^{-1} [ f_1f_2\cc{f_3} ] \|_{L^{\infty}}
 \le
 \frac{C}{t^{1/2}}\|f_1\|_{L^2} \|f_2\|_{L^2} 
 \bigl(\|f_3\|_{L^{2}}+\|\op{J}f_3\|_{L^2} \bigr).
\]
\end{lem}
%----------------
\begin{proof}
From the relation 
$\op{G} \op{U}^{-1}
 =\op{V}^{-1} \op{D}^{-1} \op{M}^{-1}$
and the estimate 
$\|\op{V}^{-1} \phi \|_{L^{\infty}} \le Ct^{1/2}\|\phi\|_{L^1}$, 
it follows that  
\begin{align*}
\|\op{G} \op{U}^{-1} [ f_1 f_2 f_3] \|_{L^{\infty}}
&\le 
Ct^{1/2} \| \op{D}^{-1} \op{M}^{-1}[f_1f_2f_3] \|_{L^1}\\
&\le
C t^{1/2} \cdot t^{-1} 
\|(\op{D}^{-1} f_1)(\op{D}^{-1} f_2)
(\op{D}^{-1} \op{M}^{-1}f_3)\|_{L^1}\\
&\le
 C t^{-1/2} \|\op{D}^{-1} f_1\|_{L^2}
 \|\op{D}^{-1} f_2\|_{L^2} 
 \|\op{D}^{-1} \op{M}^{-1}f_3\|_{L^{\infty}}\\
&=
 C t^{-1/2} \|f_1\|_{L^2} \|f_2\|_{L^2} 
 \cdot t^{1/2} \|f_3\|_{L^{\infty}}\\
&\le 
  C t^{-1/2} \|f_1\|_{L^2} \|f_2\|_{L^2} 
 \bigl(\|f_3\|_{L^{2}}+\|\op{J}f_3\|_{L^2} \bigr).
\end{align*}
We have used the inequality 
$\|f\|_{L^{\infty}}\le C t^{-1/2}\|f\|_{L^2}^{1/2} \|\op{J}f\|_{L^2}^{1/2}$
in the last line. 
The estimate for  
$\|\op{G} \op{U}^{-1} (f_1 f_2 \cc{f_3})\|_{L^{\infty}}$ 
can be shown in the same way.
\end{proof}

Next we set 
$(\op{E}^{\omega} (t) f)(y)=e^{i \omega \frac{ty^2}{2}}f(y)$ and 
$\op{A}_{\omega}(t) = {\op{V}(t)^{-1} \op{E}^{\omega-1}(t)}
 - \op{E}^{\frac{\omega-1}{\omega}}(t) \op{D}_{\omega}$.

%----------------
\begin{lem} \label{lem_tech1}
For $\omega \in \R\backslash\{0\}$, we have 
\[
 \| \op{A}_{\omega}(t) f\|_{L^{\infty}}
 \le 
 Ct^{-1/4} \|f\|_{H^1}.
\]
\end{lem}
%----------------
\begin{proof}
It follows from the relation 
${\op{V}(t)^{-1}}=\op{U}(\frac{1}{t})$ that 
\begin{align*}
 \op{V}(t)^{-1} \op{E}^{\omega-1}(t) f(\xi) 
 &= 
  \sqrt{\frac{t}{2\pi i}} \int_{\R} 
  e^{i\frac{t}{2} (\xi -y)^2} e^{i\frac{\omega-1}{2} t y^2} f(y)\,dy\\
 &= 
  e^{i\frac{\omega-1}{2\omega} t\xi^2}
 \sqrt{\frac{t}{2\pi i}} \int_{\R} 
  e^{i\frac{\omega t}{2} \left(\frac{\xi}{\omega} -y\right)^2} f(y)\,dy
  \\
 &= 
  \op{E}^{\frac{\omega -1}{\omega}}(t) \op{D}_{\omega} \op{V}(-\omega t)f(\xi).
\end{align*}
Hence we deduce from \eqref{est_W} that 
\begin{align*}
 \|\op{A}_{\omega}(t) f\|_{L^{\infty}}
 =
 \|
  \op{E}^{\frac{\omega -1}{\omega}}(t) \op{D}_{\omega}
  (\op{V}(-\omega t) -1)f
 \|_{L^{\infty}}
 \le
 C\|(\op{V}(-\omega t) -1)f\|_{L^{\infty}}
 \le
 Ct^{-1/4} \|f\|_{H^1}.
\end{align*}
\end{proof}

Now we are going to prove Lemma \ref{lem_decomp2}. 
For simplicity of exposition, we consider only the case where 
\begin{align}
 N(u,u_x)=\lambda |u_x|^2u_x +au_x^3 +b \cc{u_x^3} +c |u_x|^2\cc{u_x}.
 \label{nonlin_sp}
\end{align}
General cubic terms $N$ satisfying \eqref{cond_weak} 
(or, equivalently, $N=F+G$ with \eqref{def_F} and \eqref{def_G}) 
can be treated in the same way. Note that 
%\[
% N(z,i\xi z)
% =
% i\xi^3 (\lambda |z|^2 z  - a z^3 + b  \cc{z}^3  -c|z|^2\cc{z}), 
% \quad (z,\xi)\in \C\times \R,
%\]
%and
\[
 \nu(\xi)
=i\xi^3 \int_{0}^{2\pi}
 \left( 
  \lambda e^{i\theta}   -a e^{3i\theta}  +{b}e^{-3i\theta}-{c}e^{-i\theta} 
 \right)
 e^{-i\theta} \frac{d\theta}{2\pi }
 =
 i \lambda \xi^3 
\]
if $N$ is given by \eqref{nonlin_sp}, 
whereas 
\begin{align}
 \nu(\xi)=\lambda_1+ i(\lambda_2 -\lambda_3)\xi 
 +(\lambda_4 -\lambda_5)\xi^2 +i\lambda_6 \xi^3
\label{nu_explicit}
\end{align}
if $N=F+G$ with \eqref{def_F}, \eqref{def_G}.

First we consider the case of $l=0$. 
We put $\alpha^{(s)}=(i\xi)^s \alpha$ so that 
$\pa_x^s u= \op{M} \op{D} \op{V} \alpha^{(s)}$. 
We also set 
$\bigl( \op{M}^{\omega}(t) f \bigr)(y)=e^{i\omega \frac{y^2}{2t}}f(y)$.
Then it follows that 
\begin{align*}
 N(u,u_x)
 =& 
 \lambda \op{M} 
 |\op{D} \op{V} \alpha^{(1)}|^2 \op{D} \op{V} \alpha^{(1)}
 +
 a \op{M}^3 
  (\op{D} \op{V} \alpha^{(1)})^3\\
 &+
 b \op{M}^{-3} 
  (\cc{\op{D} \op{V} \alpha^{(1)}})^3
 +
 c \op{M}^{-1} 
 |\op{D} \op{V} \alpha^{(1)}|^2 \cc{\op{D} \op{V} \alpha^{(1)}}\\
 =& 
 \frac{\lambda}{t} \op{M} 
 \op{D} \Bigl[ |\op{V} \alpha^{(1)}|^2 \op{V} \alpha^{(1)} \Bigr]
 +
 \frac{a}{t} \op{M}^3 \op{D} \Bigl[
  (\op{V} \alpha^{(1)})^3 \Bigr]\\
 &+
 \frac{b}{t} \op{M}^{-3} \op{D} \Bigl[
  (\cc{\op{V} \alpha^{(1)}})^3 \Bigr]
 +
 \frac{c}{t} \op{M}^{-1} \op{D} \Bigl[
 |\op{V} \alpha^{(1)}|^2 \cc{\op{V} \alpha^{(1)}} \Bigr].
\end{align*}
By the relation 
$\op{G} \op{U}^{-1} \op{M}^\omega \op{D}
= \op{V}^{-1} \op{E}^{\omega-1}$, 
we have
\begin{align*}
 \op{G} \op{U}^{-1} N(u,u_x)
 =& 
 \frac{\lambda}{t} 
 \op{V}^{-1}\Bigl[|\op{V} \alpha^{(1)}|^2  \op{V} \alpha^{(1)}
 \Bigr]
 +
 \frac{a}{t} \op{V}^{-1} \op{E}^2 
   \Bigl[ (\op{V} \alpha^{(1)})^3 \Bigr]\\
 &\hspace{10mm} +
 \frac{b}{t} \op{V}^{-1} \op{E}^{-4} \Bigl[
  (\cc{\op{V} \alpha^{(1)}})^3 \Bigr]
 +
 \frac{c}{t} \op{V}^{-1} \op{E}^{-2} \Bigl[
 |\op{V} \alpha^{(1)}|^2 \cc{\op{V} \alpha^{(1)}} 
 \Bigr]\\
 =&
 \frac{\lambda}{t} i\xi^3 |\alpha|^2 \alpha 
 + \frac{a }{t} \op{E}^{\frac{2}{3}} \op{D}_{3}
    \Bigl[-i\xi^3 \alpha^3 \Bigr]\\
 &
 + \frac{b }{t} \op{E}^{\frac{4}{3}} \op{D}_{-3}
    \Bigl[i\xi^3 \cc{\alpha}^3 \Bigr]
 + \frac{c }{t} \op{E}^{2}  \op{D}_{-1}
    \Bigl[-i\xi^3 |\alpha|^2\cc{\alpha} \Bigr]
 +  \frac{\Omega_0}{t},
\end{align*}
where  
\begin{align*}
 \Omega_0
 =&
 -\lambda 
 \Bigl(
  |\alpha^{(1)}|^2  \alpha^{(1)}
  - |\op{V} \alpha^{(1)}|^2 \op{V} \alpha^{(1)} 
 \Bigr)
 +
 \lambda \bigl( \op{V}^{-1} - 1 \bigr)
 \Bigl[ |\op{V} \alpha^{(1)}|^2 \op{V} \alpha^{(1)} \Bigr]\\
 &-
 a \op{E}^{\frac{2}{3}}(t)  \op{D}_{3}
 \Bigl[
  (\alpha^{(1)})^3 - (\op{V} \alpha^{(1)})^3
 \Bigr]
 +
 a \op{A}_{3}(t) \Bigl[(\op{V} \alpha^{(1)})^3 \Bigr]\\
 &-
 b \op{E}^{\frac{4}{3}}(t)  \op{D}_{-3} 
 \Bigl[ 
  (\cc{\alpha^{(1)}})^3 - (\cc{\op{V} \alpha^{(1)}})^3
 \Bigr]
 +
 b \op{A}_{-3}(t) \Bigl[(\cc{\op{V} \alpha^{(1)}})^3 \Bigr]\\
 &-
 c \op{E}^{2}(t)  \op{D}_{-1} 
 \Bigl[
  |\alpha^{(1)}|^2 \cc{\alpha^{(1)}}  
  -
  |\op{V} \alpha^{(1)}|^2 \cc{\op{V} \alpha^{(1)}} 
 \Bigr]
 +
 c \op{A}_{-1}(t) 
 \Bigl[|\op{V} \alpha^{(1)}|^2 \cc{\op{V} \alpha^{(1)}}\Bigr].
\end{align*}
By virtue of \eqref{est_W} and Lemma \ref{lem_tech1}, we see that 
\[
 \|\Omega_0\|_{L^{\infty}} \le \frac{C}{t^{1/4}} \|u\|_{H^2}^3.
\]
Therefore we obtain \eqref{id_key} with $l=0$ by putting 
$\mu_{1,0}(\xi)=-i\frac{a}{27\sqrt{3}}\xi^2$, 
$\mu_{2,0}(\xi)=\frac{-b}{27\sqrt{3}}\xi^2$, 
$\mu_{3,0}(\xi)=c\xi^2$.

Next we consider the case of $l=1$. 
It follow from the identity \eqref{id_key2} that 
\begin{align*}
 \pa_x N(u,u_x) 
 =
  \lambda  |u_x|^2 u_{xx}
 + 3a u_x^2 u_{xx} + 3 b \cc{u_x^2 u_{xx}} +c |u_x|^2 \cc{u_{xx}}
 + \frac{1}{it}r_1, 
\end{align*}
where 
$
r_1=
(\lambda u_x +c\cc{u}_x)\bigl( (\op{J} u_x)\cc{u_x} -u_x \cc{\op{J}u_x}\bigr)
$. 
By Lemma \ref{lem_tech2}, we obtain
\[
 \|\op{G} \op{U}^{-1} r_1\|_{L^{\infty}}
 \le
 \frac{C}{t^{1/2}}\bigl( \|u\|_{H^1} + \|\op{J}u \|_{H^1} \bigr)^3.
\]
We also set 
$h_1= \lambda |u_x|^2 u_{xx}  + 3a u_x^2 u_{xx} 
+ 3 b \cc{u_x^2 u_{xx}}+c |u_x|^2 \cc{u_{xx}}$ 
so that 
\[
 \op{G}\op{U}^{-1}\pa_x N(u,u_x)
 =
 \op{G}\op{U}^{-1} h_1+\frac{1}{it}\op{G}\op{U}^{-1} r_1.
\]
Then, as in the previous case, we have
\begin{align*}
 &\op{G}\op{U}^{-1} h_1\\
 &=
 -\frac{\lambda}{t} \xi^4 |\alpha|^2 \alpha 
 + 
  \frac{3a }{t} \op{E}^{\frac{2}{3}} \op{D}_{3} 
   \Bigl[ \xi^4 \alpha^3 \Bigr]
 + 
  \frac{3b }{t} \op{E}^{\frac{4}{3}} \op{D}_{-3}
   \Bigl[ \xi^4 \cc{\alpha}^3 \Bigr]
 + 
   \frac{c }{t} \op{E}^{2} \op{D}_{-1}
   \left[ -\xi^4 |\alpha|^2\cc{\alpha} \right] 
 +  \frac{\Omega_1}{t},
\end{align*}
where 
\begin{align*}
 \Omega_1
 =&
 -\lambda 
 \Bigl(
  |\alpha^{(1)}|^2  \alpha^{(2)} 
  - 
  |\op{V} \alpha^{(1)}|^2  \op{V} \alpha^{(2)}
 \Bigr)
 +
 \lambda \Bigl(\op{V}^{-1} -1 \Bigr)
 \Bigl[|\op{V} \alpha^{(1)}|^2  \op{V} \alpha^{(2)} \Bigr]\\
 &-
  3a \op{E}^{\frac{2}{3}}(t) \op{D}_{3}
 \Bigl[
  (\alpha^{(1)})^2 \alpha^{(2)} 
  -  
  (\op{V} \alpha^{(1)})^2 \op{V} \alpha^{(1)}
  \Bigr]
 +
 3a \op{A}_{3}(t) 
  \Bigl[(\op{V} \alpha^{(1)})^2 \op{V} \alpha^{(2)} \Bigr]\\
 &-
 3b \op{E}^{\frac{4}{3}}(t)  \op{D}_{-3} 
 \Bigl[
  \cc{(\alpha^{(1)})^2 \alpha^{(2)}}
  -
  \cc{(\op{V} \alpha^{(1)})^2 \op{V} \alpha^{(2)}}
 \Bigr]
 +
 3b \op{A}_{-3}(t)
 \Bigl[\cc{(\op{V} \alpha^{(1)})^2 \op{V} \alpha^{(2)}} \Bigr]\\
 &+
 c \op{E}^{2}(t)  \op{D}_{-1} 
 \Bigl[
 |\alpha^{(1)}|^2 \cc{\alpha^{(2)}} 
 - 
 |\op{V} \alpha^{(1)}|^2 \cc{\op{V} \alpha^{(2)}} 
 \Bigr] 
 +
 c \op{A}_{-1}(t)  
  \Bigl[|\op{V} \alpha^{(1)}|^2 \cc{\op{V} \alpha^{(2)}} \Bigr].
\end{align*}
By \eqref{est_W} and Lemma \ref{lem_tech1}, we have 
\[
 \|\Omega_1\|_{L^{\infty}} \le \frac{C}{t^{1/4}} \|u\|_{H^3}^3.
\]
Therefore, by setting 
$\mu_{1,1}(\xi)=\frac{a}{27\sqrt{3}}\xi^3$, 
$\mu_{2,1}(\xi)=\frac{b}{27\sqrt{3}i}\xi^3$, 
$\mu_{3,1}(\xi)=-ic\xi^3$, 
we obtain \eqref{id_key} with $l=1$.

Finally we consider the case of $l=2$. 
From the identities \eqref{id_key1} and \eqref{id_key2}, 
it follows that 
\begin{align*}
 \pa_x^2 N(u,u_x) 
 =&
   \lambda u_x \cdot \cc{u_x} \pa_x (u_{xx}) 
+
   c \cc{u_x} \cdot u_x \pa_x(\cc{u_{xx}})
+
 (\lambda u_{xx}  +c \cc{u_{xx}})\pa_x\bigl(|u_x|^2 \bigr)\\
 &+ 
 3au_x \bigl( 2 u_{xx}^2  +  u_x \pa_x (u_{xx}) \bigr)
 + 3b \cc{u_x \bigl( 2 u_{xx}^2+   u_x \pa_x (u_{xx}) \bigr) } 
 +
 \frac{1}{it}\pa_x r_1\\
=&
 h_2 + \frac{1}{it}r_2,
\end{align*}
where 
$h_2= -\lambda |u_{xx}|^2 u_{x}  + 9a u_{xx}^2 u_{x} 
+ 9 b \cc{u_{xx}^2 u_{x}} -c |u_{xx}|^2 \cc{u_{x}}$ 
and
\begin{align*}
 r_2
 =&
 \lambda u_x 
 \big((\op{J} u_{xx})\cc{u_x} -u_{xx} \cc{\op{J}u_x} \bigr) 
 +
 c\cc{u}_x
 \bigl((\op{J} u_x)\cc{u_{xx}} -u_x \cc{\op{J} u_{xx}} \bigr)\\
&+
  (\lambda u_{xx}  +c \cc{u_{xx}})
 \bigl( (\op{J} u_x)\cc{u_x} -u_x \cc{\op{J}u_x} \bigr) 
 +
 3a u_x \bigl( u_x \op{J}u_{xx} -(\op{J}u_x) u_{xx}\bigr)\\
 &
 - 
 3b \cc{u_x \bigl( u_x \op{J}u_{xx} - (\op{J}u_x) u_{xx} \bigr)}
 +
 \pa_x r_1.
\end{align*}
We deduce as before that
\begin{align*}
 &\op{G}\op{U}^{-1} h_2\\
 &= 
 -\frac{\lambda}{t}  i\xi^5 |\alpha|^2 \alpha 
 + \frac{9a}{t} \op{E}^{\frac{2}{3}} \op{D}_{3} 
    \Bigl[i\xi^5 \alpha^3 \Bigr]
 + \frac{9b}{t} \op{E}^{\frac{4}{3}} \op{D}_{-3}
    \Bigl[-i\xi^5 \cc{\alpha}^3\Bigr]
  - \frac{c}{t} \op{E}^{2} \op{D}_{-1}
   \Bigr[-i\xi^5 |\alpha|^2\cc{\alpha} \Bigr]
 +  \frac{\Omega_2}{t}
\end{align*}
with
\[
 \|\Omega_2\|_{L^{\infty}} \le \frac{C}{t^{1/4}} \|u\|_{H^3}^3.
\]
We also have 
\[
 \|\op{G} \op{U}^{-1} r_2\|_{L^{\infty}}
 \le
 \frac{C}{t^{1/2}}\bigl( \|u\|_{H^2} + \|\op{J}u\|_{H^2} \bigr)^3
\]
by virtue of Lemma \ref{lem_tech2}. 
Now we set 
$\mu_{1,2}(\xi)=\frac{ai}{27\sqrt{3}}\xi^4$, 
$\mu_{2,2}(\xi)=\frac{b}{27\sqrt{3}}\xi^4$, 
$\mu_{3,2}(\xi)=-c\xi^4$. 
 Then we arrive at \eqref{id_key} with $l=2$. 
This completes the proof of Lemma \ref{lem_decomp2}.
\qed

%%%%%%%%%%%%%%%%%%%%%%%%%%%%%%%%%%%%%%%%%%%%%%%%%%%%%%%%%%%%%%%%%%%%%%%%%%%%%%%
\medskip
\subsection*{Acknowledgments}
The authors are grateful to  Professor Soichiro Katayama for his useful 
conversations on this subject. 
The work of H.~S. is supported by Grant-in-Aid for Scientific Research (C) 
(No.~25400161), JSPS.

%%%%%%%%%%%%%%%%%%%%%%%%%%%%%%%%%%%%%%%%%%%%%%%%%%%%%%%%%%%%%%%%%%%%%%%%%%%%%%%

%%%%%%%%%%%%%%%%%%%%%%%%%%%%%%%%%%%%%%%%%%%%%%%%%%%%%%%%%%%%%%%%%%%%%%%%%%%%%%%

\begin{thebibliography}{99}


%----------
\bibitem{Chi}H.~Chihara, 
{\em
Local existence for the semilinear Schr\"odinger equations in one space 
dimension,} 
J. Math. Kyoto Univ. {\bf 34} (1994), no.2, 353--367. 
%----------


%----------
\bibitem{Chr} D.~Christodoulou, 
{\em
Global solutions of nonlinear hyperbolic equations for small initial data,} 
Comm. Pure Appl. Math. {\bf 39} (1986), no.2, 267--282.
%-----------




%----------
\bibitem{De}J.-M.~Delort, 
{\em
Minoration du temps d'existence pour l'\'equation de Klein-Gordon 
non-lin\'eaire en dimension 1 d'espace,} 
Ann. Inst. H. Poincar\'e Anal. Non Lin\'eaire {\bf 16} (1999), no.5, 563--591. 
%----------




%----------
\bibitem{HN002} N.~Hayashi and P.~I.~Naumkin, 
{\em 
Asymptotic behavior in time of solutions to the derivative nonlinear 
Schr\"odinger equation revisited,} 
Discrete Contin. Dynam. Systems {\bf 3} (1997), no.3, 383--400.
%----------



%----------
\bibitem{HN003} N.~Hayashi and P.~I.~Naumkin, 
{\em 
Asymptotics for large time of solutions to the nonlinear Schr\"odinger and 
Hartree equations,} 
Amer. J. Math. {\bf 120} (1998), no.2, 369--389. 
%----------




%----------
\bibitem{HN1} N.~Hayashi and P.~I.~Naumkin, 
{\em 
Large time behavior of solutions for derivative cubic nonlinear Schr\"odinger 
equations without a self-conjugate property,} 
 Funkcial. Ekvac. {\bf 42} (1999), no.2, 311--324. 
%----------



%----------
\bibitem{HN2} N.~Hayashi and P.~I.~Naumkin, 
{\em 
Asymptotics of small solutions to nonlinear Schr\"odinger equations with 
cubic nonlinearities,} 
Int. J. Pure Appl. Math. {\bf 3} (2002), no.3, 255--273.
%----------



%----------
\bibitem{HN004} N.~Hayashi and P.~I.~Naumkin, 
{\em 
Large time behavior for the cubic nonlinear Schr\"odinger equation,} 
Canad. J. Math. {\bf 54} (2002), no.5, 1065--1085.
%----------




%----------
\bibitem{HN005} N.~Hayashi and P.~I.~Naumkin, 
{\em 
On the asymptotics for cubic nonlinear Schr\"odinger equations,} 
Complex Var. Theory Appl. {\bf 49} (2004), no.5, 339--373.
%----------




%----------
\bibitem{HN006} N.~Hayashi and P.~I.~Naumkin, 
{\em
Asymptotics of odd solutions for cubic nonlinear Schr\"odinger equations,} 
J. Differential Equations {\bf 246} (2009), no.4, 1703--1722.
%----------




%----------
\bibitem{HN007} N.~Hayashi and P.~I.~Naumkin, 
{\em 
Global existence for the cubic nonlinear Schr\"odinger equation in lower order 
Sobolev spaces,} 
Differential Integral Equations {\bf 24} (2011), no.9--10, 801--828. 
%----------




%----------
\bibitem{HN008} N.~Hayashi and P.~I.~Naumkin, 
{\em 
Logarithmic time decay for the cubic nonlinear Schr\"odinger equations,} 
Int. Math. Res. Not. IMRN 2015, no.14, 5604--5643. 
%----------






%----------
\bibitem{HNP} N.~Hayashi, P.~I.~Naumkin and P.-N.~Pipolo, 
{\em 
Smoothing effects for some derivative nonlinear Schr\"odinger equations,} 
 Discrete Contin. Dynam. Systems {\bf 5} (1999), no.3, 685--695.
%----------


%----------
\bibitem{HNS} N.~Hayashi, P.~I.~Naumkin and H.~Sunagawa, 
{\em 
On the Schr\"odinger equation with dissipative nonlinearities of derivative 
type,} 
 SIAM J. Math. Anal. {\bf 40} (2008), no.1, 278--291.
%----------



%----------
\bibitem{HNU} N.~Hayashi, P.~I.~Naumkin and H.~Uchida, 
{\em 
Large time behavior of solutions for derivative cubic nonlinear Schr\"dinger 
equations,} 
Publ. Res. Inst. Math. Sci. {\bf 35} (1999), no.3, 501--513. 
%----------



%----------
\bibitem{HO} N.~Hayashi and T.~Ozawa, 
{\em
Remarks on nonlinear Schr\"odinger equations in one space dimension,} 
Differential Integral Equations {\bf 7} (1994), no. 2, 453--461.
%----------



%----------
\bibitem{Hor1} L.~H\"ormander, 
{\em The lifespan of classical solutions of nonlinear hyperbolic equations,} 
Springer Lecture Notes in Math. {\bf 1256} (1987), 214--280.
%----------



%----------
\bibitem{Hor2} L.~H\"ormander, 
{\em Remarks on the Klein-Gordon equation,} 
Journ\'ees \'equations aux deriv\'ees partielles 
(Saint Jean de Monts, 1--5 juin, 1987), Exp. No.I, 9pp., 
\'Ecole Polytech., Palaiseau, 1987.
%----------




%----------
\bibitem{J} F.~John, 
{\em
Existence for large times of strict solutions of nonlinear wave equations in 
three space dimensions for small initial data,} 
Comm. Pure Appl. Math. {\bf 40} (1987), no.1, 79--109.
%----------



%----------
\bibitem{KT} S.~Katayama and Y.~Tsutsumi, 
{\em 
Global existence of solutions for nonlinear Schr\"odinger equations in one 
space dimension,} 
Comm. Partial Differential Equations {\bf 19} (1994), no.11-12, 1971--1997. 
%----------




%----------
\bibitem{KPV} C.~Kenig, G.~Ponce and L.~Vega, 
{\em 
Small solutions to nonlinear Schr\"odinger equations,}
Ann. Inst. H. Poincar\'e Anal. Non Lin\'eaire {\bf 10} (1993), no.3, 
255--288. 
%----------



%----------
\bibitem{Kla} S.~Klainerman, 
{\em 
The null condition and global existence to nonlinear wave equations,} 
in ``Nonlinear Systems of Partial Differential Equations in Applied Mathematics, Part 1," 
Lectures in Applied Math., vol. 23, pp. 293--326. AMS, Providence, RI (1986).
%-----------



%----------
\bibitem{LS}
C.~Li and H.~Sunagawa,
{\em 
On Schr\"odinger systems with cubic 
dissipative nonlinearities of derivative type}, 
preprint, 2015 (arXiv:1507.07617  [math.AP]).
%----------



%----------
\bibitem{N} P.~I.~Naumkin, 
{\em 
Cubic derivative nonlinear Schr\"odinger equations,} 
 SUT J. Math. {\bf 36} (2000), no.1, 9--42. 
%----------



%----------
\bibitem{N2} P.~I.~Naumkin, 
{\em 
The dissipative property of a cubic non-linear Schr\"odinger equation,} 
Izv. Math. {\bf 79} (2015), no.2, 346--374. 
%----------





%----------
\bibitem{O} T.~Ozawa, 
{\em 
On the nonlinear Schr\"odinger equations of derivative type,} 
Indiana Univ. Math. J. {\bf 45} (1996), no.1, 137--163. 
%----------







%----------
\bibitem{Su} H.~Sunagawa, 
{\em 
Lower bounds of the lifespan of small data solutions to the nonlinear 
Schr\"{o}dinger equations,} 
Osaka J. Math. {\bf 43} (2006), no.4, 771--789.
%----------






%----------
\bibitem{Tone} S.~Tonegawa, 
{\em 
Global existence for a class of cubic nonlinear Schr\"odinger equations 
in one space dimension,}
Hokkaido Math. J. {\bf 30} (2001), no.2, 451--473.
%----------





%----------
\bibitem{Tsu} Y.~Tsutsumi, 
{\em 
The null gauge condition and the one-dimensional nonlinear Schr\"odinger 
equation with cubic nonlinearity,} 
Indiana Univ. Math. J. {\bf 43} (1994), no.1, 241--254. 
%----------



\end{thebibliography}
\end{document}